\documentclass[12pt]{amsart}

\usepackage[margin=1.0in]{geometry}
\usepackage{amsmath}
\usepackage{amssymb}
\usepackage{setspace}
\usepackage[all]{xy}

\newcommand{\bp}{\mathbf p}
\newcommand{\bq}{\mathbf q}
\newcommand{\rat}{\dashrightarrow}
\renewcommand{\P}{\mathbb P}
\newcommand{\PGL}{\text{PGL}}
\newcommand{\C}{\mathbb C}

\newcommand{\cstar}{($\ast$)}

\newcommand{\git}{/\!\!/}

\newtheorem{theorem}{Theorem}
\newtheorem*{theorem*}{Theorem}

\newtheorem{lemma}[theorem]{Lemma}

\newtheorem*{conjecture*}{Conjecture}
\newtheorem*{question*}{Question}

\theoremstyle{definition}

\newtheorem*{remarks}{Remarks}

\DeclareMathOperator{\Pico}{Pic^0}
\newcommand{\Coh}{\text{Coh}}
\DeclareMathOperator{\Cr}{Cr}
\newcommand{\cO}{\mathcal O}
\newcommand{\D}{D}

\begin{document}

\title{Derived-equivalent rational threefolds}
\author{John Lesieutre} 
\address{Department of Mathematics\\
  MIT\\
  77 Massachusetts Avenue\\
  Cambridge, MA 02139, USA} \email{johnl@math.mit.edu}

\thanks{The author was supported by an NSF Graduate Research
  Fellowship under Grant \#1122374.}

\begin{abstract}
  We describe an infinite set of smooth projective threefolds that
  have equivalent derived categories but are not isomorphic, contrary
  to a conjecture of Kawamata.  These arise as blow-ups of \(\P^3\) at
  various configurations of \(8\) points, which are related by Cremona
  transformations.
\end{abstract}

\maketitle

\section{Introduction}

For a smooth projective variety \(X\), let \(\D(X) = D^b \Coh(X)\)
denote the bounded derived category of coherent sheaves on \(X\).  The
derived category contains a great deal of information about the
variety \(X\): if \(Y\) is another variety whose derived category is
equivalent to that of \(X\), then \(X\) and \(Y\) have the same
dimension and Kodaira dimension, and, if \(X\) is of general type,
they are are birational~\cite{kawamataequiv}.  Extending a result of
Bridgeland and Maciocia~\cite{bridgelandmaciocia}, Kawamata proved
that if \(X\) is a smooth projective surface, there are only finitely
many other smooth projective surfaces \(Y\) (up to isomorphism), with
\(D(X)\) and \(D(Y)\) equivalent as triangulated categories. He asked
whether this property might hold in all dimensions~\cite[Conjecture
1.5]{kawamataequiv}.

We observe here that there are threefolds for which this is not the
case.  Let \(\bp\) denote an ordered \(8\)-tuple of distinct points in
\(\P^3\), and let \(X_\bp\) be the blow-up of \(\P^3\) at the points
of \(\bp\).
\begin{theorem*}
  There is an infinite set \(W\) of configurations of \(8\) points in
  \(\P^3\) such that if \(\bp\) and \(\bq\) are distinct elements of
  \(W\), then \(D(X_\bp) \cong D(X_\bq)\) but \(X_\bp\) and \(X_\bq\)
  are not isomorphic.
\end{theorem*}
\noindent The example boils down to three basic observations, made precise in
Lemmas 1, 2, and 3.
\begin{enumerate}
\item \(X_\bp\) and \(X_\bq\) are isomorphic if and only if \(\bp\)
  and \(\bq\) coincide, up to permutation and an automorphism of
  \(\P^3\).
\item If \(\bq\) can be obtained from \(\bp\) by a sequence of
  standard Cremona transformations centered at \(4\)-tuples from among
  the points of \(\bp\), then \(X_\bp\) and \(X_\bq\) are connected by
  a sequence of flops of rational curves with normal bundle \(\cO(-1)
  \oplus \cO(-1)\), and so \(D(X_\bp) \cong D(X_\bq)\).
\item The orbit of a sufficiently general configuration \(\bp\) of
  \(8\) points under standard Cremona transformations is infinite.
\end{enumerate}

There are several classes of higher-dimensional varieties for which
\(D(X)\) has been shown to determine the isomorphism class of \(X\),
up to finitely many possibilities.  These include abelian varieties
over \(\C\) ~\cite{orlovabelian},\cite{favero}, toric
varieties~\cite{kawamatatoricii}, varieties with \(K_X\) ample, and
Fano varieties~\cite{bondalorlov}.  It also known that the number of
isomorphism classes of varieties with a given derived category is at
most countable~\cite{toen}. Further discussion of this problem can be
found in~\cite{favero} and \cite{rouquier}.

The example is based on the action of Cremona transformations on
configurations of points in \(\P^3\), as investigated by A.\ Coble.
Most of the results we will need can be found in Dolgachev and
Ortland's account of Coble's work~\cite{dolgachev}.  We provide
self-contained proofs, with references to the more general theory.

It is worth noting that these examples do not pose any problems for
the Kawamata-Morrison cone conjecture for klt Calabi-Yau
pairs~\cite{totaro}.  Through \(8\) general points in \(\P^3\) there
is a pencil of quadrics, with base locus a degree \(4\) curve through
the points.  If \(\Delta\) is the sum of two generic quadrics in this
pencil, then \((X,\Delta)\) is a dlt pair with \(K_X + \Delta\)
numerically trivial.  However, there is no choice of \(\Delta\) for
which \(K_X+\Delta\) is numerically trivial and the pair is klt, for
the divisor obtained by blowing up the curve in the base locus has log
discrepancy \(0\).  If the points are specialized to the base locus of
a two-dimensional net of quadrics, the pair \((X,\Delta)\) can be made
klt, but the configuration is insufficiently general for our construction,
and indeed Prendergast-Smith has demonstrated that the cone conjecture
holds for this class of examples~\cite{prendergast}.  Note also that
although the blow-up of \(\P^3\) at \(7\) very general points is a
weak Fano variety, the blow-up at \(8\) is not.

\section{The example}

For the rest of this section, \(\bp\) denotes an ordered \(8\)-tuple
of distinct points in \(\P^3\), with \(\pi_\bp : X_\bp \to \P^3\) the
blow-up of the points of \(\bp\).  Write \(E_{i}\) for the
exceptional divisors of \(\pi_\bp\) and \(H\) for the pullback to
\(X_\bp\) of the hyperplane class on \(\P^3\).

\begin{lemma}[\cite{dolgachev}, Ch.\ V.1]
\label{isomorphic}
The blow-ups \(X_\bp\) and \(X_\bq\) are isomorphic if and only if
\(\bp\) and \(\bq\) coincide, up to an automorphism of \(\P^3\) and
permutation of the points.
\end{lemma}
\begin{proof}
  Suppose that \(\phi : X_\bp \to X_\bq\) is an isomorphism.  The
  restriction \(\P^3 \setminus \bp \cong X_\bp \setminus \bigcup E_i
  \to X_\bq \to \P^3\) defines a rational map \(\psi : \P^3 \rat
  \P^3\), with indeterminacy locus contained in \(\bp\) and hence
  \(0\)-dimensional.  Its inverse is likewise regular outside a
  \(0\)-dimensional set, contained in \(\bq\).  But any rational map
  \(\psi : \P^3 \rat \P^3\) for which \(\psi\) and \(\psi^{-1}\) both
  have \(0\)-dimensional indeterminacy sets is in fact an
  automorphism (e.g.\ by Theorem 1.1 of~\cite{bayraktarcantat}).  
  Now, \(\psi \circ \pi_\bp\) contracts \(E_i\) to a point, and so
  \(\pi_\bq\) must contract \(\phi(E_i)\) to a point.  Consequently
  \(\psi(\bp) = \bq\), so the configurations differ by a permutation
  and automorphism, and \(\phi\) identifies the exceptional divisors.
\end{proof}

The basic ingredient in constructing other blow-ups which are
derived-equivalent to a given one is the action of the standard
Cremona transformation \(\Cr : \P^3 \rat \P^3\), defined by
\([X_0:X_1:X_2:X_3] \mapsto [X_0^{-1}:X_1^{-1}:X_2^{-1}:X_3^{-1}]\).
A resolution of this rational map is as follows.
\[
\xymatrix{
 & Y \ar[dl]_p \ar[dr]^{p^\prime} & \\
X \ar@{-->}[rr]^{\overline{\Cr}} \ar[d]_\pi &&  X^\prime \ar[d]^{\pi^\prime} \\
\P^3 \ar@{-->}[rr]^{\Cr} && \P^3
}
\]
Here \(\pi\) blows up the four standard coordinate points.  The strict
transforms on \(X\) of the six lines \(\ell_{ij}\) between two of
these points are smooth rational curves with normal bundle \(\cO(-1)
\oplus \cO(-1)\).  These are flopped by \(\overline{\Cr}\); \(p\)
blows up these curves to divisors isomorphic to \(\P^1 \times \P^1\),
which are then contracted along the other ruling by \(p^\prime\).  The
indeterminacy locus of \(\overline{\Cr} : X \rat X^\prime\) is the
union of these curves. The map \(\pi^\prime\) then blows down the
strict transforms of the four planes through three of the four
original points.

Suppose that \(\bp\) is a configuration of \(8\) points (regarded now
as a point on the configuration space \((\P^3)^8 \git \PGL(4)\)), and
four points are chosen from among \(\bp\) satisfying the following
condition:
\begin{itemize}
\item[\cstar] No other point of \(\bp\) lies on any plane defined by
  three of the four chosen points.
\end{itemize}
Condition \cstar{} implies that the Cremona transformation centered at
the four given points is defined, as these are not coplanar.  It also
guarantees that no point of \(\bp\) is on one of the contracted
divisors or one of the lines in the indeterminacy locus.  We can
define a new configuration \(\bq\) by making a Cremona transformation
centered at the four chosen points, and moving the remaining four
points under that transformation: if \(p_j\) is one of the chosen
points, then \(q_j\) is defined as the image of the plane through the
other three points (which is contracted by \(\pi^\prime\)), while if
\(p_j\) is not a chosen point, then \(q_j\) is just the image of
\(p_j\) under the Cremona transformation. The configuration \(\bq\) is
defined up to choice of coordinates, and no two points of \(\bq\) are
infinitely near, since no point of \(\bp\) is on one of the contracted
divisors.  After blowing up the points of \(\bp\) and \(\bq\), there
is a rational map \(\overline{\Cr} : X_\bp \rat X_\bq\) which flops
six rational curves of normal bundle \(\cO(-1) \oplus \cO(-1)\).

We will say that \(\bp\) and \(\bq\) are \emph{Cremona equivalent} if
there exists a sequence of Cremona transformations centered at \(4\)-tuples
from among the points which sends the configuration \(\bp\) to
\(\bq\), and satisfies condition \cstar{} at each step.  The
\emph{Cremona orbit} of \(\bp\) is the set of all configurations that
are Cremona equivalent to \(\bp\).  If \(\bp\) is a very general
configuration of points, then any sequence of Cremona transformations
will automatically satisfy \cstar{}; however, we prefer not to make
any blanket generality assumption on \(\bp\) at this stage, as it will
be useful to consider \(8\)-tuples in slightly special configurations.

\begin{lemma}
\label{sqms}
If \(\bp\) and \(\bq\) are Cremona equivalent, then \(\D(X_\bp) \cong
\D(X_\bq)\).
\end{lemma}
\begin{proof}
  Assumption \cstar{} on the Cremona transformations implies that there is
  a sequence of rational maps \(X_\bp = X_{\bp_0} \rat X_{\bp_1} \rat
  \cdots \rat X_{\bp_n} = X_\bq\), where each \(X_{\bp_j} \rat
  X_{\bp_{j+1}}\) flops six curves.  The lemma is then a consequence
  of a fundamental result of Bondal and Orlov~\cite[Theorem
  4.3]{bondalorlov}: if \(X\) and \(X^+\) are threefolds, and \(\phi :
  X \rat X^+\) is the flop of a rational curve with normal bundle
  \(\cO(-1) \oplus \cO(-1)\), then \(\D(X) \cong \D(X^+)\).
\end{proof}

There is a subtlety here in that each of these rational maps flops six
disjoint curves, while the theorem of Bondal and Orlov is usually
stated for the flop of a single curve.  Each map can be factored into
a sequence of six disjoint flops, but the intermediate varieties
encountered are no longer projective. However, the needed result is
valid without assuming \(X\) and \(X^+\) are projective~\cite[Remark
11.24ii]{huybrechts}.

\begin{lemma}[cf.~\cite{dolgachev}, Ch.\ VI]
\label{distinct}
A very general configuration \(\bp\) of \(8\) points has infinite
Cremona orbit.
\end{lemma}

\begin{proof}
  We will define a sequence of configurations in which \(\bp_{n+1}\)
  is obtained from \(\bp_n\) by a Cremona transformation centered at
  four points, satisfying condition \cstar{}, and such that the
  \(\bp_n\) are all distinct.  Let \(C = C_0\) be a smooth genus \(1\)
  curve in \(\P^3\) obtained as the complete intersection of two
  smooth quadrics.  Choose \(p_{5}\), \(p_{6}\), \(p_{7}\), and
  \(p_8\) to be the four points of intersection of some generic
  hyperplane \(H\) with \(C\), and then choose \(p_1\), \(p_2\),
  \(p_3\), and \(p_{4}\) to be very general points of \(C\); this
  guarantees that if \(4d - \sum_{i=1}^8 m_i = 0\), the class
  \(dH\vert_C - \sum_{i=1}^8 m_i p_i \in \Pico(C)\) is nonzero unless
  \(m_1 = \cdots = m_{4} = 0\) (cf.~\cite[Lemma 2.4]{lu2}).

  Let \(\bar{C}\) denote the strict transform of \(C\) on
  \(X_{\bp_0}\), and say that a prime divisor \(D \sim dH -
  \sum_{i=1}^8 m_i E_i\) on \(X_{\bp_0}\) is a \emph{root divisor} if
  it does not contain \(\bar{C}\) and satisfies \(4d - \sum_{i=1}^8
  m_i = 0\).  If \(D\) is a root divisor, then \(D \cdot \bar{C} = 0\)
  and so \(D\) and \(\bar{C}\) are disjoint.  This implies that
  \(dH\vert_C - \sum_{i=1}^8 m_i p_i \in \Pico(C)\) is trivial, and so
  \(m_1 = \cdots = m_4 =0\).  If \(m_5\), \(m_6\), \(m_7\), and
  \(m_8\) are not all equal, then some \(m_i\) is greater than \(d\),
  and the corresponding class cannot be effective.  Hence the only
  root divisor on \(X_{\bp_0}\) is the strict transform of the plane
  through the last four points, with numerical class \(H -
  \sum_{i=5}^8 E_i\).

  We now inductively define \(\bp_{n+1}\) by making Cremona
  transformation centered at the first four points of \(\bp_n\), and
  then cyclically permuting the points so that \(p_1\) comes last.  At
  each step, we show that the first four points of \(\bp_n\) satisfy
  condition \cstar{}.  These transformations induce rational maps
  \(X_{\bp_n} \rat X_{\bp_{n+1}}\); let \(\bar{C}_n\) denote the
  strict transform of \(\bar{C}\) on \(X_{\bp_n}\), and \(C_n\) its
  image in \(\P^3\).  Since \(C_0\) is the intersection of two
  quadrics, and Cremona transformations preserve quadrics through the
  \(8\) points, each \(C_n\) is also an intersection of two quadrics.
  In particular, no \(\bar{C}_n\) can be contained in the
  indeterminacy locus of \(X_{\bp_{n}} \rat X_{\bp_{n+1}}\).

  Suppose that \(\bp\) is a configuration of points and that \(\bq\)
  is the configuration obtained by making a standard Cremona
  transformation centered at the first four points of \(\bp\).  If
  \(D\) is a divisor in the class \(dH - \sum_{i=1}^8 m_i E_{i}\) on
  \(X_\bp\), then its strict transform on \(X_\bq\) has class
  \(d^\prime H^\prime - \sum_{i=1}^8 m_i E_i^\prime\), where
  \(d^\prime = 3d - \sum_{i=1}^4 m_i\), \(m_i^\prime =
  2d+m_i-\sum_{i=1}^4 m_i\) for \(1 \leq i \leq 4\), and \(m_i^\prime
  = m_i\) for \(5 \leq i \leq 8\)~\cite{lu2}.  Write \(M : N^1(X_\bp) \to
  N^1(X_\bq)\) for the corresponding linear map.

  Let \(M_\sigma =PM\), where \(M\) is as above and \(P\) is the
  permutation matrix which permutes the exceptional divisors by moving
  the first one to last.  If \(D\) is a divisor on \(X_{\bp_i}\), its
  strict transform on \(X_{\bp_{i+1}}\) has class \(M_\sigma(D)\). The
  map \(M_\sigma^n : N^1(X_{\bp_0}) \to N^1(X_{\bp_n})\) preserves the
  effective cones, as well as the property that \(4d - \sum_{i=1}^8
  m_i = 0\).

  Suppose that \(D \sim dH - \sum_{i=1}^8 m_i E_i\) is a root divisor
  on \(X_{\bp_n}\) (i.e.\ prime, not containing \(\bar{C}_n\), and
  with \(4d - \sum_{i=1}^8 m_i = 0\)).  Then the strict transform of
  \(D\) on \(X_{\bp_0}\) is a root divisor as well.  Since there is a
  unique such divisor on \(X_{\bp_0}\), we conclude that \(D\) has
  numerical class \(M_\sigma^n(H - \sum_{i=5}^8 E_i)\) on
  \(X_{\bp_n}\).  It is straightforward to check that the classes
  \(M_\sigma^n(H - \sum_{i=5}^8 E_i)\) are all distinct; the argument
  is indicated in Lemma~\ref{linalg}, which is postponed until the end
  of this section.

  It follows that condition \cstar{} holds for the configuration
  \(\bp_n\): if any four points \(p_{j_1}\), \dots, \(p_{j_4}\) of
  \(\bp_n\) were coplanar, then \(H - \sum_{i=1}^4 E_{j_i} \) would be
  a root divisor on \(X_{\bp_n}\), which is possible only if \(n=0\)
  and the points in question are \(p_5\), \(p_6\), \(p_7\), and
  \(p_8\).  In particular, the Cremona transformation defining
  \(\bp_{n+1}\) is well-defined for every \(n\), and this gives an
  infinite sequence of configurations connected by Cremona
  transformations, all satisfying \cstar{}.  Since the degrees of the
  classes \(M_\sigma^n(H- \ \sum_{i=5}^8 E_i)\) grow unboundedly,
  there are infinitely many distinct configurations among the
  \(\bp_n\), even up to permutation of the points.  As the Cremona
  orbit is infinite for the special configuration \(\bp_0\), it is
  also infinite for very general configurations.
\end{proof}

\begin{proof}[Proof of Theorem]
  Let \(\bp\) be a very general \(8\)-tuple of points in \(\P^3\), and
  let \(W\) be the Cremona orbit of \(\bp\).  By Lemma~\ref{distinct},
  \(W\) contains infinitely many distinct configurations of points,
  even up to permutations.  Lemma~\ref{sqms} then shows that the
  blow-ups \(X_{\bq}\) for \(\bq \in W\) are all derived-equivalent,
  but by Lemma~\ref{isomorphic} no two are isomorphic.
\end{proof}

\begin{remarks}
  In fact this construction can easily be generalized to
  configurations of \(k \geq 8\) points. Permutations of the points,
  together with the Cremona transformation at the first four, generate
  the action of a Coxeter group of type \(T_{2,4,k-4}\) on the
  configuration space \((\P^3)^k \git \PGL(4)\) by birational maps.
  This group is infinite as soon as \(k \geq 8\), and sufficiently
  general configurations have infinite orbits.  This is explored in
  detail in \cite[Ch.\ VI]{dolgachev}.  The construction in
  Lemma~\ref{distinct} simply follows the orbit of a single special
  configuration under the iteration of a Coxeter element in this
  group.

  The property of having only a finite orbit under Cremona
  transformations is quite special: a complete classification of such
  configurations in \(\P^2\) has been given by Dolgachev and
  Cantat~\cite{dolgachevcantat}.  The case in which \(\bp\) is a
  configuration which is Cremona equivalent to itself under some
  nontrivial sequence of Cremona transformations is also quite
  interesting; in this case there is a pseudoautomorphism \(\phi :
  X_\bp \rat X_\bp\)~\cite{perronizhang},
  inducing an autoequivalence \(D(X_\bp) \to D(X_\bp)\).
\end{remarks}

\begin{lemma}
\label{linalg}
  For any \(m\) and \(n\), the classes \(M_\sigma^m(H - \sum_{i=5}^8
  E_i)\) and \(M_\sigma^n(H - \sum_{i=5}^8 E_i)\) are distinct.
\end{lemma}
\begin{proof}
  Explicitly, the matrix \(M_\sigma\) is given with respect to the
  basis \(H,E_i\) as
\setlength{\arraycolsep}{2pt}
\[
M_\sigma = \begin{tiny}\left(
\begin{array}{rrrrrrrrr}
3 & 1 & 1 & 1 & 1 & 0 & 0 & 0 & 0 \\ 
-2 & -1 & 0 & -1 & -1 & 0 & 0 & 0 & 0 \\ 
-2 & -1 & -1 & 0 & -1 & 0 & 0 & 0 & 0 \\ 
-2 & -1 & -1 & -1 & 0 & 0 & 0 & 0 & 0 \\ 
0 & 0 & 0 & 0 & 0 & 1 & 0 & 0 & 0 \\ 
0 & 0 & 0 & 0 & 0 & 0 & 1 & 0 & 0 \\ 
0 & 0 & 0 & 0 & 0 & 0 & 0 & 1 & \phantom{-}0 \\ 
0 & \phantom{-}0 & \phantom{-}0 & \phantom{-}0 & \phantom{-}0 & \phantom{-}0 & \phantom{-}0 & \phantom{-}0 & 1 \\ 
-2 & 0 & -1 & -1 & -1 & 0 & 0 & 0 & 0
\end{array}\right)
\end{tiny}
\]
Let \(M_\sigma = SJS^{-1}\) be the Jordan decomposition.  A
computation shows that \(J\) has a \(3 \times 3\) Jordan block
associated to the eigenvalue \(1\).  One can compute the coefficients of
\(H-\sum_{i=5}^8 E_i\) in the Jordan basis as the entries of \(S^{-1}
\left( H-\sum_{i=5}^8 E_i \right)\), and observe \(H-\sum_{i=5}^8
E_i\) has nonzero coefficients for two of the generalized eigenvectors
in the nontrivial Jordan block.  It follows that the powers
\(M_\sigma^n(H-\sum_{i=5}^8 E_i)\) are all distinct. 
\end{proof}
For a more enlightened perspective on this calculation from the point
of view of Coxeter groups, we refer to~\cite[\S 2]{perronizhang}
and~\cite[Thm.\ 2.2]{mcmullen}.

\section*{Acknowledgements}

I benefited from discussions of this example with a number of people,
including Roberto Svaldi, Eric Bedford, Paul Seidel, and especially my
advisor, James M\textsuperscript{c}Kernan, who suggested numerous
improvements to early drafts.

\singlespacing

\nocite{mukai}

\bibliographystyle{amsplain}
\bibliography{zrefs}

\providecommand{\bysame}{\leavevmode\hbox to3em{\hrulefill}\thinspace}
\providecommand{\MR}{\relax\ifhmode\unskip\space\fi MR }
\providecommand{\MRhref}[2]{%
  \href{http://www.ams.org/mathscinet-getitem?mr=#1}{#2}
}
\providecommand{\href}[2]{#2}
\begin{thebibliography}{10}

\bibitem{toen}
Mathieu Anel and Bertrand To{\"e}n, \emph{D\'enombrabilit\'e des classes
  d'\'equivalences d\'eriv\'ees de vari\'et\'es alg\'ebriques}, J. Algebraic
  Geom. \textbf{18} (2009), no.~2, 257--277.

\bibitem{bayraktarcantat}
Turgay Bayraktar and Serge Cantat, \emph{Constraints on automorphism groups of
  higher dimensional manifolds}, pre-print (2013).

\bibitem{bondalorlov}
A.~Bondal and D.~Orlov, \emph{Derived categories of coherent sheaves},
  Proceedings of the {I}nternational {C}ongress of {M}athematicians, {V}ol.
  {II} ({B}eijing, 2002) (Beijing), Higher Ed. Press, 2002, pp.~47--56.

\bibitem{bridgelandmaciocia}
Tom Bridgeland and Antony Maciocia, \emph{Complex surfaces with equivalent
  derived categories}, Math. Z. \textbf{236} (2001), no.~4, 677--697.

\bibitem{dolgachevcantat}
Serge Cantat and Igor Dolgachev, \emph{Rational surfaces with a large group of
  automorphisms}, J. Amer. Math. Soc. \textbf{25} (2012), no.~3, 863--905.

\bibitem{dolgachev}
Igor Dolgachev and David Ortland, \emph{Point sets in projective spaces and
  theta functions}, Ast\'erisque (1988), no.~165, 210 pp. (1989).

\bibitem{favero}
David Favero, \emph{Reconstruction and finiteness results for {F}ourier-{M}ukai
  partners}, Adv. Math. \textbf{230} (2012), no.~4-6, 1955--1971.

\bibitem{huybrechts}
D.~Huybrechts, \emph{Fourier-{M}ukai transforms in algebraic geometry}, Oxford
  Mathematical Monographs, The Clarendon Press Oxford University Press, Oxford,
  2006.

\bibitem{kawamataequiv}
Yujiro Kawamata, \emph{{$D$}-equivalence and {$K$}-equivalence}, J.
  Differential Geom. \textbf{61} (2002), no.~1, 147--171.

\bibitem{kawamatatoricii}
\bysame, \emph{Derived categories of toric varieties {II}}, Michigan Math. J.
  \textbf{62} (2013), no.~2, 353--363.

\bibitem{lu2}
Antonio Laface and Luca Ugaglia, \emph{Elementary {$(-1)$}-curves of {$\Bbb
  P^3$}}, Comm. Algebra \textbf{35} (2007), no.~1, 313--324.

\bibitem{mcmullen}
Curtis~T. McMullen, \emph{Dynamics on blowups of the projective plane}, Publ.
  Math. Inst. Hautes \'Etudes Sci. (2007), no.~105, 49--89.

\bibitem{mukai}
Shigeru Mukai, \emph{Counterexample to {H}ilbert's fourteenth problem for the
  3-dimensional additive group}, RIMS Kyoto preprint \textbf{1343} (2001).

\bibitem{orlovabelian}
D.~O. Orlov, \emph{Derived categories of coherent sheaves on abelian varieties
  and equivalences between them}, Izv. Ross. Akad. Nauk Ser. Mat. \textbf{66}
  (2002), no.~3, 131--158.

\bibitem{perronizhang}
Fabio Perroni and De-Qi Zhang, \emph{Pseudo-automorphisms of positive entropy
  on the blowups of products of projective spaces}, pre-print (2011).

\bibitem{prendergast}
Arthur Prendergast-Smith, \emph{The cone conjecture for some rational elliptic
  threefolds}, Math. Z. \textbf{272} (2012), no.~1-2, 589--605.

\bibitem{rouquier}
Rapha{\"e}l Rouquier, \emph{Cat\'egories d\'eriv\'ees et g\'eom\'etrie
  birationnelle (d'apr\`es {B}ondal, {O}rlov, {B}ridgeland, {K}awamata et
  al.)}, Ast\'erisque (2006), no.~307, Exp. No. 946, viii, 283--307,
  S{\'e}minaire Bourbaki. Vol. 2004/2005.

\bibitem{totaro}
Burt Totaro, \emph{The cone conjecture for {C}alabi-{Y}au pairs in dimension
  2}, Duke Math. J. \textbf{154} (2010), no.~2, 241--263.

\end{thebibliography}

\end{document}